\documentclass[12pt,leqno,twoside]{amsart}

\usepackage[latin1]{inputenc}
\usepackage[T1]{fontenc}
\usepackage[backref,colorlinks=true, pdfstartview=FitV, linkcolor=blue, citecolor=blue, urlcolor=blue]{hyperref}

\usepackage{amstext,amsmath,amscd, bezier,indentfirst,amsthm,amsgen,enumerate, geometry}
\usepackage[all,knot,arc,import,poly]{xy}
\usepackage{amsfonts,color, soul}  
\usepackage{amssymb}
\usepackage{latexsym}
\usepackage{epsfig}
\usepackage{graphicx}
\usepackage{srcltx}
\usepackage{enumitem, comment}

\newtheorem{theorem}{Theorem}[section]
\newtheorem{corollary}[theorem]{Corollary}
\newtheorem{lemma}[theorem]{Lemma}
\newtheorem{proposition}[theorem]{Proposition}

\theoremstyle{remark}

\theoremstyle{remark}

\theoremstyle{definition}

\theoremstyle{remark}

\theoremstyle{remark}

\theoremstyle{remark}

\renewcommand{\Box}{\square}    

\newcommand{\cal}{\mathcal}

\renewcommand{\int}{{\mathrm{int}}}
\newcommand{\gen}{{\mathrm{gen}}}
\newcommand{\pol}{{\mathrm{pol}}}

\newcommand{\isol}{{\mathrm{is}}}
\newcommand{\Sing}{{\mathrm{Sing\hspace{1pt} }}}

\newcommand{\mult}{{\mathrm{mult}}}

\newcommand{\grad}{\mathop{\mathrm{grad}}\nolimits}

\newcommand{\rk}{\mathrm{rank\hspace{2pt}}}
\newcommand{\e}{\varepsilon}
\newcommand{\spec}{{\mathrm{spec}}}
\newcommand{\m}{\setminus}
\newcommand{\fin}{\hspace*{\fill}$\Box$\vspace*{2mm}}

\newcommand{\tF}{F^{\pitchfork}}

\newcommand{\tmu}{\mu^{\pitchfork}}

\newcommand{\cA}{{\cal A}}

\newcommand{\cH}{{H}}

\newcommand{\cR}{{\cal R}}

\newcommand{\bC}{{\mathbb C}}

\newcommand{\bP}{{\mathbb P}}

\newcommand{\bZ}{{\mathbb Z}}

\newcommand{\bV}{{\mathbb V}}

\begin{document}

\title[Polar degree]
 {Polar degree of hypersurfaces with 1-dimensional singularities}

\author{\sc Dirk Siersma}  

\address{Institute of Mathematics, Utrecht University, PO
Box 80010, \ 3508 TA Utrecht, The Netherlands.}

\email{D.Siersma@uu.nl}

\author{\sc Mihai Tib\u{a}r}

\address{Univ. Lille, CNRS, UMR 8524 - Laboratoire Paul Painlev\'e, F-59000 Lille, France.}

\email{mihai-marius.tibar@univ-lille.fr}

\thanks{The authors thank the Mathematisches Forschungsinstitut Obewolfach for supporting this research project through the Research in Pairs program, and acknowledge the support of the Labex CEMPI grant (ANR-11-LABX-0007-01). }


\date{\today}






\begin{abstract}
 We prove a formula for the polar degree of  projective hypersurfaces in terms of the Milnor data of the singularities, extending to 1-dimensional singularities the Dimca-Papadima result for isolated singularities.  We discuss the semi-continuity of the polar degree in deformations, and we classify the homaloidal cubic surfaces with 1-dimensional singular locus. Some open questions are pointed out along the way.
 \end{abstract}

%

\maketitle

\setcounter{section}{0}

\sloppy

\section{introduction}\label{s:introd}

For any projective hypersurface $V \subset \bP^{n}$, defined by a homogeneous polynomial $f : \bC^{n+1} \to \bC$ of degree $d$,  the \emph{polar degree} is defined as the topological degree of the gradient map, also known as the Gauss map:
\begin{equation}\label{eq:grad}
\grad f : \bP^{n}\m \Sing (V) \to  \bP^{n}.
\end{equation}

The polar degree depends only on $V$  and not on the defining polynomial $f$, as conjectured by Dolgachev \cite{Do} and proved by Dimca and Papadima in \cite{DP}. One therefore denotes it by $\pol (V)$.

 The concept of polar degree goes back to 1851 when Hesse  studied hypersurfaces with vanishing Hessian  \cite{Hes51, Hes59},  which is equivalent to $\pol(V) =0$,  and to Gordon and Noether \cite{GN76} (see also \S\ref{ss:notcone} for more comments).

\smallskip

The gradient maps \eqref{eq:grad} of polar degree equal to $1$ are examples of \emph{Cremona transformations} since they are birational maps.  The corresponding hypersurfaces $V$ were called \emph{homaloidal}, and Dolgachev \cite{Do} found the list of all (reduced) projective plane curves.

\smallskip

In the beginning of the 2000's, whereas the algebraic approach was dominant before that date, Dimca and Papadima \cite{DP} gave the following topological interpretation:
\emph{For any projective hypersurface $V$, if $\cH$ is a general hyperplane  with respect to $V$,  then the homology of the affine part $H_{n-1}(V\m \cH)$
 is concentrated in dimension $n-1$, and: 
\begin{equation}\label{eq:red}
  \pol(V) = \rk H_{n-1}(V\m \cH).
\end{equation}}

The classification of all homaloidal hypersurfaces with \emph{isolated singularities} was carried out by Huh \cite{Huh} and
confirmed a conjecture stated by Dimca and Papadima \cite{DP, Di2} that there are no  homaloidal hypersurfaces with isolated singularities  besides the smooth quadric and the plane curves found by Dolgachev. Huh  \cite{Huh} proves and uses the bound:
\begin{equation}\label{eq:huhbound}
\pol(V) \ge  \mu^{\left< n-2\right> }_p(V) : = \mu_p(V \cap \cH),
\end{equation}
where the Milnor number $ \mu_p(V \cap \cH)$ is $>0$ as soon as $p\in \Sing V$.  This holds at any $p$ such that $V$ is not a cone of apex $p$. 

 More recently,  the authors together with  Steenbrink  classified in \cite{SST} the hypersurfaces with isolated singularities and polar degree 2, confirming Huh's conjectural list \cite{Huh}.  The finiteness of the range of $(n,d)$ in which there may exist hypersurfaces with isolated singularities and polar degree $k>2$ has been also proved in \cite{SST}.

 \medskip
 
  Still for isolated singularities, Dimca and Papadima \cite{DP}  had shown the formula: 
  \begin{equation}\label{eq:polmu}
  \pol(V) = (d-1)^{n} - \sum_{p\in \Sing V} \mu_{p}(V).
\end{equation}
 which allows to compute the polar degree in terms of the Milnor data of the singular points.
 
 \medskip

In this paper we  consider hypersurfaces with 1-dimensional singular locus. We extend the formula \eqref{eq:polmu}  and  compute the polar degree from the Milnor data of the singularities (Theorem  \ref{t:dimsing1}):
\begin{equation*} 
 \pol (V) = (d-1)^{n} -  \sum_{p\in \Sigma^{\isol}}\mu_{p} -\sum_{i=1}^{r} c_i  \mu_{i}^{\pitchfork} + (-1)^{n}\sum_{q\in Q} (\chi(\cA_{q}) - 1) 
\end{equation*}
by using the study of the  hypersurfaces with 1-dimensional singular locus in  \cite{ST-vhom} and in earlier papers, see \cite{Si-cambridge}. The notations are explained in \S\ref{s:DP1dim}.

\medskip

We use the semi-continuity of the polar degree in deformations (Proposition  \ref{t:polsemi-cont}) in order
to compare the polar degree of V  with 1-dimensional singularity  with its deformation $V_{t}$ (of Yomdin type) to a hypersurface with isolated singularities  (Corollary \ref{c:yomdin}). Then $\pol(V_{t})$  may  serve as upper bound  for the polar degree of $V$.
As a consequence, we derive a Lefschetz type inequality for the slicing with a generic hyperplane $\cH_{gen}$:
\begin{equation*}
\pol(V) \le (d-1) \pol ( V \cap \cH_{gen}).
\end{equation*}


We highlight the concept of \emph{special point} of a hypersurface introduced in \cite{ST-polcycles}. These are points $p$ where the complex link of $(V,p)$ is non-trivial, it appears that they are finitely many, and they provide lower bounds for the polar degree. We show that, in the 1-dimensional singularity case,  one can detect them by the  Milnor number jump of the transversal singularity type. 

We  treat cubic surfaces in \S\ref{s:examples}. We compute all polar degrees in a topological way  and prove that there are 3 homaloidal cubic surfaces with non-isolated singularities. 

   We discuss homaloidal hypersurfaces with 1-dimensional singularities and transversal type $A_1$,  and we state and discuss the question of the existence of hypersurfaces with a 1-dimensional singularities and polar degree equal to zero, in this manner coming back to Hesse's problem cited in the beginning.

\section{Formula for $\pol(V)$ in case of 1-dimensional singularities}\label{s:DP1dim}

Let  $V\subset \bP^{n}$  be a projective hypersurface with singular locus $\Sing (V)$ of dimension $\le 1$ and 
 $\Sing (V) = \Sigma^{c} \cup \Sigma^{\isol}$, where $\Sigma^{c}$ is a non-degenerate curve with irreducible components $\Sigma^{c}_{i}$, $i=1, \ldots, r$, and  $\Sigma^{\isol}$ is the finite set of isolated singularities.

  Each curve branch $\Sigma^{c}_i$ of $\Sing (V)$ has a generic transversal type, of transversal Milnor fibre $\tF_i$ and Milnor number denoted by $\tmu_i$.

   Each branch  $\Sigma^{c}_i$ contains a finite set $Q_i$ of points where the transversal type is not the generic one, and which we have called  \emph{special points} (see also \S\ref{ss:specialp} for a more general definition). We denote by $\cA_q$ the local Milnor fibre of the hypersurface germ $(V, q)$ for $q\in Q := \cup_{i=1}^{r}Q_i$, and by   $\tilde \Sigma^{c}_i$   the normalisation of  $\Sigma^{c}_i$.
   
   At each point $q\in Q_i$ there are finitely many locally irreducible branches of the germ $(\Sigma^{c}_i, q)$,  we denote by $\gamma_{i,q}$ their number and let $\gamma_{i} := \sum_{q\in Q_i}\gamma_{i,q}$. In other words, $\gamma_{i}$ is equal to the number of ``punctures'' in $\tilde \Sigma_i \setminus \tilde Q$.

 This setting has been studied in \cite{ST-vhom}, in particular a formula for  $\chi (V)$ has been found. It depends on the  isolated singularities of $V$, on the special nonisolated singularities,  on the topology of the curve components of $\Sing (V)$, and on the transversal singularity type of each such curve component.
 
Under the above notations, our following formula  generalises the Dimca-Papadima formula \eqref{eq:polmu} for isolated singularities to the case of a 1-dimensional singular set. 

\begin{theorem}\label{t:dimsing1} Let $V\subset \bP^{n}$ be a hypersurface of degree $d$ with a 1-dimensional singular set. Then:
\begin{equation}\label{eq:dpgen}
 \pol (V) = (d-1)^{n} -  \sum_{p\in \Sigma^{\isol}}\mu_{p} -\sum_{i=1}^{r} c_i  \mu_{i}^{\pitchfork} + (-1)^{n}\sum_{q\in Q} (\chi(\cA_{q}) - 1) 
\end{equation}
where $c_i =  2g_{i}  +\gamma_{i} +(d+1)\deg \Sigma^{c}_{i}-2$, where
$g_i$ is the genus of the normalization $\tilde \Sigma^{c}_i$ of $\Sigma^{c}_{i}$, and where
$\deg \Sigma^{c}_{i}$ denotes the degree of $\Sigma^{c}_{i}$ as a reduced curve.
\end{theorem}

\begin{proof}
The Euler characteristic $\chi(V)$  has been computed in \cite{ST-vhom} by using a local pencil of hypersurfaces $V_\e := \{ f_\e = f + \e h_d = 0\}$ of degree $d= \deg f$, where  $h_d$ is a general homogeneous polynomial.  
  Let $A := \{ f= h_d = 0\}$ denote the axis of the pencil.  
   
   Let
 $\bV_\Delta := \{ (x, \e) \in \bP^{n+1} \times \Delta \mid f + \e h_d = 0\}$ 
denote the total space of the pencil, where $V_0 := V \subset \bP^{n} \times \{ 0\}$ and $\Delta$ is a small enough disk centred at $0\in \bC$ such that $V_\e$ is nonsingular for all $\e \in \Delta^*$.  The existence of small enough disks $\Delta$
is ensured\footnote{see e.g. \cite[Prop. 2.2]{ST-bettimax} for a detailed explanation.} by the genericity of $h_d$.
 Note that $\bV_\Delta$ retracts to $V$.
 
\smallskip

We have introduced in \cite{ST-vhom} the \emph{vanishing homology}\footnote{For the  treatment of the  general singular setting by using sheaf cohomology, see  the more recent paper \cite{MPT}.} of projective hypersurfaces $V$ with 
 $\dim \Sing (V) =1$,  defined as:
\[  H^{\curlyvee}_{*}(V) := H_*(\bV_\Delta, V_\e;\bZ)\]
 and in its study we have established the following Euler characteristic formula for $\chi(\bV_\Delta, V_\e)$ which equals $\chi(V) - \chi(V_\e)$, since  $\bV_\Delta$ retracts to $V$:

\[ \chi(V) - \chi(V_{\e}) = (-1)^{n} \sum_{i=1}^{r} (2g_{i} +\gamma_{i} +\nu_{i} - 2) \mu_{i}^{\pitchfork} - \sum_{q\in Q} (\chi(\cA_{q}) - 1) + (-1) ^{n} \sum_{p\in \Sigma^{\isol}}\mu_{p}
\] 
where $\nu_i := \# A\cap \Sigma^{c}_i$ is the number of axis points, namely $\nu_i = \int(\{h_{d}=0\},\Sigma^{c}_{i}) = d  \deg \Sigma^{c}_{i}$.

Since $V_\e$ is a nonsingular hypersurface in $\bP^{n}$, its Euler characteristic is that of the smooth hypersurface of degree $d$  in $\bP^{n}$, namely:
$$\chi^{n,d} :=
 n+1 - \frac{1}{d} [ 1 + (-1)^{n} (d-1)^{n+1} ].$$

Let now $H$ be a generic hyperplane with respect to the canonical Whitney stratification of $V$. 
Since  $V \cap H$ has only isolated singularities, one has the well-known formula:

\[ \chi (V \cap H) = \chi^{n-1,d} + (-1)^{n-1} \sum_{a \in \Sing (V \cap H)} \mu_a(V \cap H),\]
where $\Sing (V \cap H) = \Sigma^{c} \cap H$.
The number $\# \Sigma^{c} \cap H$ of intersection points 
is  then $\sum_{i=1}^{r}\deg \Sigma^{c}_{i}$, and $\mu_a(V \cap H)$, for  $a\in \Sigma^{c}_{i}$, is precisely the transversal Milnor number   $\mu_{i}^{\pitchfork}$.

To prove our claimed formula for $\pol(V)$ we use the above formulas and the  relation  \eqref{eq:red} of \cite{DP} in the form:
 \[\chi (V \setminus H) = 1 + (-1)^{n-1} \pol (V)\]
where $\chi (V \setminus H) = \chi (V) - \chi (V \cap H)$ by the additivity of $\chi$. 
\end{proof}


The above formula is useful as soon as the local topological information is available. We show this is some examples
taken from \cite{ST-vhom}.

\smallskip
\noindent
\cite[Ex 7.1a]{ST-vhom}: $V := \{ x^2z + y^2 w = 0\} \subset \bP^3$.  Then $\Sing(V)$ is a projective line with  transversal type $A_1$ and with two special points of type $D_{\infty}$, the Milnor fibre of which is homotopy equivalent to $S^2$. Applying formula \eqref{eq:dpgen} we  get  $\pol(V) = 8 - 0 - (0+2+4-2) - 2 = 2$.

\smallskip
\noindent
\cite[Ex 7.1b]{ST-vhom}: $V := \{x^2z + y^2 w + t^3= 0\} \subset\bP^4$.   Then $\Sing(V)= \bP^1$  is also a projective line but here with  transversal type $A_2$ and with two special points having Milnor fibre $S^3 \vee S^3$. From  formula  \eqref{eq:dpgen} we  get   $\pol(V) = 16 - 0 - (0+2+4-2) \cdot 2 - 4 = 4$.

\section{Special points and lower bounds}\label{s:lowerbound}\label{ss:decomp}

If one looks for hypersurfaces with small polar degree, for instance homaloidal, it is useful to have lower bounds for the polar degree in terms of the singularities of $V$ or its dual. An example is Huh's bound \eqref{eq:huhbound} for isolated singularities,
which has been used in \cite{Huh} and in \cite{SST} to determine all hypersurfaces with isolated singularities 
which have polar degree 1 or 2.  A more general bound has been found in \cite{ST-polcycles}.  

\subsection{Quantisation of the polar degree, after \cite{ST-polcycles}}
 
Let us fix a Whitney stratification $\cal W$ of $V \subset \bP^n$. This depends only on the reduced structure of $V$. 
We consider hyperplanes $\cH$ which are stratified transversal to all strata of $V$, except at finitely many  points.
One says that the hyperplane $H= \{ l=0\}$ is \emph{admissible} iff its non-transversality locus consists of isolated points only, and if  the affine polar locus $\Gamma(l, f)$ has dimension $\le 1$.

If $\cH$ has an isolated non-transversality at $q \in V$ then the linear function $l : \bC^{n} \to \bC$ defining $\cH$  near $q$ has a stratified  isolated singularity at $q$.
 Consequently, it's local Milnor-L\^e fibre  $B_{\e}\cap V \cap \{l=s\}$, for some $s$ close enough to $l(q)$, has the homotopy type of a bouquet of spheres of dimension $n-2$, cf L\^{e}'s  results \cite{Le}. 
We denote by $\alpha_q(V, \cH)$ it's Milnor-L\^e number. In case $H = \cH_{\gen}$ is a general hyperplane through $q$, then  $\alpha_q(V,\cH_{\gen})$ is the Milnor number of the \emph{complex link} of  $V$ at $q$, denoted 
 by $\alpha_q(V)$, and we have  $\alpha_q(V,\cH) \ge \alpha_q(V)$.

Let us denote by $\alpha(V,\cH)$ the sum $\sum_{q}\alpha_q(V, \cH)$ for all points $p$ of the  non-transversality locus of the admissible hyperplane $H$.
We have shown in \cite{ST-polcycles}:
\begin{theorem}\cite[Theorem 5.4]{ST-polcycles}\label{t:quantis}
  If $H$ is an admissible hyperplane,  then there is the following decomposition of the polar degree: 
$$\pol (V) = \alpha(V,\cH) + \beta(V,\cH)$$
 where $\beta(V,\cH)$ is related to the vanishing cycles of  the affine polynomial map which is the restriction of  $f: \bC^{n+1} \to \bC$ to the affine hyperplane plane $l=1$. 
 \fin
\end{theorem}
 
Let us consider now the special case of hypersurfaces $V$ with 1-dimensional singular locus.
In this case, the numbers $\alpha_q(V,H)$ can be described by sectional Milnor numbers, as follows:
\begin{itemize}
\item[(i)]
For isolated singular points $p \in V$ we have $\alpha_p(V,\cH) = \mu_p(V \cap \cH)$ by definition.

\item[(ii)]
For points $p$ on the 1-dimensional singular set, we have the formula proved in \cite{ST-gendefo}: 
\begin{equation}\label{eq:mdefect}
\alpha_p(V,\cH)  =   \mu(V \cap \cH, p)  -  \sum_{j} \mu(V \cap \cH_s,p_j),
\end{equation}
where the points $p_{j}$ are the singular points of the restriction  ${f_p}_{|H_{s}}$ in $B_{\e}\cap V\cap H_s$,
where  $f_p =0$ is  a reduced local equation for $V$. 
The number  $\alpha_p(V,\cH)$ was called the \emph{Milnor number jump} at $p$ for the family of functions ${f_p}_{|\cH_{s}}$. Even if $\cH_{0}$ may be not the most generic at $p$,  for $s\not= 0$ the hyperplane $\cH_{s}$ is a locally generic slice of  any branch $\Sigma^c_i$, and so the above equality  reads:
\begin{equation}\label{eq:mdefect2}
\alpha_p(V,\cH)  =   \mu(V \cap \cH, p)  -  \sum_{i} \mult_{p} \Sigma^c_{i} \cdot \mu_{i}^{\pitchfork},
\end{equation}
where $\mu_{i}^{\pitchfork}$ is the generic transversal Milnor number and does not depend on the choice of the point of the irreducible component $\Sigma^c_{i}$.  Here the sum is taken over all local branches at $p$.
\end{itemize}

\subsection{The special points of $V$}\label{ss:specialp}
For any singular projective hypersurface $V$, we say after \cite{ST-polcycles} that  $p \in V$  a {\em special point of V} if $\alpha_p(V) > 0$. It turns out  that the set of special points $V_{\spec}$ is finite.

 In case  $\dim_{p} \Sing (V) =1$, the set of special points of $V$  consists of the isolated singularities of $V$  together with the set $Q$ of points $p$  where the generic transversal Milnor number is jumping (see  \S\ref{s:DP1dim}), which is equivalent to the inequality $\alpha_p(V,\cH) > 0$.
Indeed, it is well-known that $\alpha_p(V,\cH) = 0$ implies that  $\Sing (V)$ is smooth  at $p$  and that $V \cap \cH_s$ is a $\mu$-constant local family of hypersurfaces.  This is equivalent to A'Campo's ``non-splitting principle'' \cite{AC}.

It is known (see \cite[Remark 4.3]{ST-polcycles}) that the set of admissible hyperplanes for $f$ containing a fixed point $p\in \Sing V$ is a Zariski-open subset of the set of all hyperplanes through $p$.  The following useful lower bound then holds in general, as a consequence of Theorem \ref{t:quantis}:

\begin{corollary}\cite[Corollary 6.6]{ST-polcycles}\label{c:admissibleandspecial}
Let $V\subset \bP^{n}$ be a projective hypersurface which is not a cone of apex  $p$.  Then:
\begin{equation}\label{eq:newbounds}
 \pol(V)  \ge \alpha_p(V).
\end{equation}
In particular, if $V$ is not a cone, then:
\begin{equation}\label{eq:newbounds2}
 \pol(V)  \ge \max_{q\in V_{\spec}}\alpha_q(V).
\end{equation}
\fin
\end{corollary}
We shall apply this lower bound result in case of homaloidal hypersurfaces, see \S\ref{ss:homaloidal}.

\section{Semi-continuity of $\pol(V)$ under deformations} \label{s:semi}

The following result is general,  it holds for any singular locus and whatever $\pol(V)$ might be, including the case $\pol(V)= 0$. This could be folklore, but we didn't find a precise reference.
\begin{proposition}\label{t:polsemi-cont}
The polar degree is lower semi-continuous in deformations of fixed degree $d$.  More precisely,  if $f_s$ is a deformation of $f_{0} := f$ of constant degree,   then  $\pol (V_s) \ge \pol (V)$ for  $s\in \bC$ close enough to 0,
 where $V_{s} := \{f_{s}=0\}$.
\end{proposition}
\begin{proof}
Let $b\in \bP^{n}$ be a regular value for $\grad f : \bP^{n}\m \Sing V \to  \bP^{n}$ and let  $(\grad f)^{-1}(b) = \{a_1,\cdots,a_k\}$, $k\ge 0$.
There exist disjoint compact neighborhoods $U_i$ of $a_i$ and $U'$ of $b$ such that $\grad f : (U_i,a_i )\to (U',b)$ is a diffeomorphism. Next take $s$ so close to 0 such that $\grad f_s | U_i$ are still diffeomorphisms, and that $\grad f_s (U_i)$ still contains $b$ in its interior. Let $W = \cap_{i=1}^k \grad f_s(U_i)$ and $Z_i = (\grad f_s)^{-1}(W) \cap U_i$. The restriction 
$ \grad f_s : \bigsqcup_{i} Z_i \to W $
is a diffeomorphism on each component $Z_i$, and has topological degree $\pol(V)$. Moreover $b$ is still a regular value for this restriction, but perhaps not anymore for the full map   $\grad f _s: \bP^{n}\m \Sing V_s \to  \bP^{n}$. Arbitrarily close to $b$ there exist points $b'$ which are  regular values for $\grad f_s$.  Then the  number of counter-images $\#(\grad f_s)^{-1}(b')$  is $\pol (V_s)$ and $(\grad f_s)^{-1}(b')$ contains at least one point in each $Z_i$.  This shows the inequality $\pol (V_s) \ge \pol (V)$.
\end{proof}

\medskip

\subsection*{Yomdin type families and polar degree}\label{ss:scpencil}

Let  $V = \{ f=0\}$ have at most 1-dimensional singularities, and  let us consider the particular deformation\footnote{Known as the \emph{Yomdin series} in the local context.}  $f_s = f +s  l^d$ of degree $d$,  where $H = \{l=0\}$  is a hyperplane such that $H \cap V$  has isolated singularities only.  
By direct computation, one can show that for  generic $s \ne 0$ (actually except of a finite number of values of $s\in \bC$),  the singular locus of $V_s = \{ f_s =0\}$ is the set $H\cap \Sing V$, and therefore $\Sing V_s$ consists of  isolated singular points.
We may then show:
 \smallskip

\begin{corollary}\label{c:yomdin}
If $V$ has at most 1-dimensional singularities, and  if $H_{gen}$ is a generic hyperplane, then:
\begin{equation}\label{eq:polslicing}
\pol(V) \le \pol(V_s) = (d-1) \pol ( V \cap \cH_{gen})
\end{equation}
and
 \begin{equation}\label{eq:condtion1dim}
\sum_i \tilde d_{i}\tmu_i \le (d-1)^{(n-1)},
\end{equation}
where $ \tilde d_{i} :=\int(\Sigma^c_i,\cH_{gen})$ denotes the global intersection number. 
\end{corollary}

\begin{proof}
By \cite{Yo}, one has the following formula for the local ``Yomdin series'' $g_{N} = g +s  l^N$ of a function germ $g: (\bC^{n},0) \to (\bC, 0)$ with 1-dimensional singular locus $\Sigma = \cup_{j}\Sigma_{j}$,  where $l$ is a general linear form, see also \cite{Si-cambridge}:

\begin{equation}\label{eq:yomdinmu}
\mu(g_{N}) = b_{n-1}(g) - b_{n-2}(g)  + N \sum_{j} d_j \tmu_j
\end{equation}
where $\mu(g_{N})$ is the Milnor number of  $g_{N}$, where $b_{n-1} (g)$ and $b_{n-2}(g)$ are the respective Betti numbers of the local Milnor fibre of $g$, and where $\tmu_j$ is the transversal Milnor number of the local branch $\Sigma_{j}$. The sum is taken over the branches $\Sigma_{j}$ of $\Sigma$,  and
 $d_j = \int_0(\Sigma_j,\{l=0\})$ is the  intersection multiplicity  of $\{l=0\}$ and $\Sigma_j$ with reduced structure.

  This formula was proved initially for ``high enough $N$'' depending on the polar ratios of the discriminant of the map $(l,g)$. It was shown in \cite{Si-series} that actually it holds for $N$ greater or equal to the highest polar ratio. Moreover, from the  proof of this formula in  \cite[pag 187]{Si-series}, one deduces the following statement:
  
\begin{lemma}\label{l:yomdin}
 If the polar locus $\Gamma(l,g)$ is empty, then formula \eqref{eq:yomdinmu} holds for any $N\ge 2$.
\fin
\end{lemma}

Let us remark that if $\Gamma(l,g) =\emptyset$ then $b_{n-1}(g) = 0$ and 
 $b_{n-2}(g)= \mu(g_{|l=0})$. Thus formula \eqref{eq:yomdinmu} becomes: 
\begin{equation}\label{eq:yomdinmu0}
\mu(g_{N}) =  - \mu(g_{|l=0})  + N \sum_{j} d_j \tmu_j
\end{equation}

  We will apply Lemma \ref{l:yomdin} and formula \eqref{eq:yomdinmu0} at some point $p\in H_{\gen}\cap \Sigma^c_i$,   for the following data:  $g:= f_p$, where  $f_p=0$ is a local equation of $V$ at $p$,  $N=d$ and  $g_{d} := f_s = f +s  l^d$ as defined above. 
  
  We remark that the genericity of $H_{\gen}$ with respect to $\Sing V$ implies that  the local polar locus  $\Gamma_{p}(l,f)$ is empty, and therefore Lemma \ref{l:yomdin} holds indeed for our data at $p$. Moreover, 
 since  $H_{\gen}$ is generic,  $\Sigma^c_{i}$ is smooth  at $p \in H_{gen}\cap \Sigma^c_i$, 
 thus in  formula \eqref{eq:yomdinmu0} we have a single term in the sum, a single multiplicity  $d_{i} = 1$, a single transversal Milnor number $\tmu_i$,  and  $b_{n-2}(f_p)= \tmu_i$. Therefore  formula \eqref{eq:yomdinmu0} reduces to:
\begin{equation}\label{eq:line}
\mu_p(V_s)  = (d-1) \tmu_i  .
\end{equation}
Taking the sum over all points $p\in  \Sing V_{s}$, we get:
 \[\sum_{p\in  H_{\gen}\cap \Sing V} \mu_{p}(V_{s}) = (d-1) \sum_i \tilde d_{i}\tmu_i .\]

Inserting this in the Dimca-Papadima formula \eqref{eq:polmu} for $V_{s}$,
 we get:
  \sloppy 
\begin{equation}\label{eq:polgeneric}   
\begin{array}{l}
\pol(V_s) = (d-1)^n - (d-1) \sum_i \tilde d_{i}\tmu_i  \\
 =(d-1) \big[ (d-1)^{(n-1)}- \sum_p \mu_p(V \cap \cH_{gen}) \big] 
  =(d-1) \pol ( V \cap \cH_{gen}).
\end{array}
\end{equation}
By applying now the semi-continuity result  Proposition \ref{t:polsemi-cont}, we obtain \eqref{eq:polslicing}.

From  \eqref{eq:polgeneric} we also get $\pol ( V \cap \cH_{gen})= (d-1)^{n-1} - \sum_i \tilde d_{i}\tmu_i \ge 0$, hence our 
claimed inequality \eqref{eq:condtion1dim} follows too.
\end{proof}


\section{Classification results, and questions}\label{s:examples}

\subsection{Cubic surfaces}

 We list here the polar degrees of all reduced cubic surfaces, based on the classification 
done by Bruce and Wall \cite{BW} with a singularity theory approach.  In case of {\em isolated singularities}, we 
 give the number of singularities and their types. By using the Dimca-Papadima formula \eqref{eq:polmu} one gets:\\
 
 \noindent
$\pol (V) = 8$: the smooth cubic\\
$\pol (V) = 7: A_1.$\\
$\pol (V) = 6:  2A_{1} \mbox{ or } A_2$.\\
$\pol (V) = 5: 3A_1 \mbox{ or }  A_1A_2 \mbox{ or } A_3$.\\
$\pol (V) = 4: 4A_1 \mbox{ or }  A_22A_1 \mbox{ or }  A_3A_1 \mbox{ or }  2A_2 \mbox{ or }  A_4 \mbox{ or }  D_4.$\\
$\pol (V) =3: A_32A_1 \mbox{ or } A_12A_2 \mbox{ or } A_4A_1 \mbox{ or } A_5 \mbox{ or } D_5 $.\\
$\pol(V) = 2: 3A_2 \mbox{ or }  A_5A_1 \mbox{ or }  E_6$. \\
$\pol(V) = 1$: no homaloidal surfaces.\\
$\pol(V) = 0: \tilde E_6$, which is a cone.

\medskip

Next, the Bruce-Wall  classification \cite{BW} of {\em  irreducible} cubics with {\em nonisolated singularities} contains:\\

\noindent
 (CN) cone over a nodal curve, \\
 (CC) cone over a cuspidal curve\\
 
\noindent  both with $\pol (V) = 0$ because they are cones,
  and two other cases, for which we use our formula \eqref{eq:dpgen} to compute $\pol(V)$:\\

\noindent  
(E1) $x_0^2x_2 + x_1^2x_3   \; ; \; \pol(V)=2 $\\
(E2)  $x_0^2x_2 +x_0 x_1 x_3  + x_1^3\; ;  \; \pol(V) = 1$,\\

\noindent
where the singular set is a projective line with  two special points of type $D_{\infty}$ in the first case,  and a single  special point of type $J_{2,\infty}$ in the second case\footnote{For the type notation, see  \cite{Si-ILs}.}.

The irreducible surface (E2) with the  $J_{2,\infty}$ point is also mentioned in \cite[Section  3.2]{CRS} as $Y(1,2)$, related to a rational scroll surface and in \cite[Example 4.7]{CRS} as a sub-Hankel surface.

\medskip
Among the {\em reducible} cubics there are only the following three cases with non-zero polar degree:

\begin{itemize}
\item[(QP)] The union of a smooth quadratic with a general hyperplane: $\pol(V) = 2$,
\item[(QT)] The union of a smooth quadratic with a tangent hyperplane: $\pol(V) = 1, $
\item[(CP)] The union of a quadratic cone and a general hyperplane: $\pol(V)=1$.
\end {itemize}

For these computations one can use the union formula \eqref{eq:union} for the union of   two hypersurfaces $V_i\subset \bP^{n}$:

\begin{equation}\label{eq:union} \pol (V_1 \cup V_2)  = \pol(V_1) + \pol(V_2) + (-1)^n [\chi(V_1 \cap V_2 \setminus H) -1, ]
\end{equation}
where $H$ be a generic hyperplane with respect to $V$. This is  a consequence of \eqref{eq:red}
and of the inclusion-exclusion principle for the Euler-characteristic. It was observed in several papers, e.g. \cite[Cor. 4.3]{FM}.  

\smallskip

All the other reducible cubics are cones  and thus have  $\pol(V)=0$.
In particular, there are only three homaloidal cubic surfaces,  all with nonisolated singularities.
   
   \subsection{Homaloidal singularities with transversal type $A_1$}\label{ss:homaloidal}
A necessary condition for $V$ to be homaloidal is that   $\alpha_q(V) = 1$  for any special point $p\in V$, cf Corollary \ref{c:admissibleandspecial}.
This restriction tells for instance that all isolated singularities of $V$ must be of type $A_k$, as observed in \cite {Huh} in case of isolated singularities only, but the argument holds at any isolated singular point, cf Corollary \ref{c:admissibleandspecial}, whatever the other singularities of $V$ might be.

\smallskip
For instance, a hypersurface $V$ with 1-dimensional singular locus can be viewed (by slicing) as a 1-parameter  family
of hypersurfaces with isolated singularities. 
Therefore, a natural  classification question is: \emph{what are the 1-parameter families of hypersurfaces with isolated singularities with Milnor jumps equal to 1?} 

 It is known that in certain generic transversal types (e.g. $\tilde{E}_6$, $\tilde{E}_7$, or $\tilde{E}_8$) this jump do not exist. Let us look now to the case of the  transversal type $A_1$, more precisely the subclass of ``smooth singular locus and generic transversal type $A_1$''.


\begin{proposition}
Let $V \subset \bP^{n}$ be a hypersurface with singular locus $\Sing (V) = \Sigma^{c}\sqcup \Sigma^{\isol}$ such that $\Sigma^{c}$ is a smooth projective line with transversal type $A_1$, and $\Sigma^{\isol}$ is a finite set of isolated points. 

 Then
 $\alpha_{p}(V)=1$ if and only if the hypersurface germ $(V, p)$  is a $J_{k,\infty}$ singularity, i.e has local equation: 
$y^2 (y-x^k) + z_1^2 + \cdots+ z_{n-2}^2= 0$.

In particular, if $V$ is homaloidal, then all its  special  points  on $\Sigma^{c}$ must be of type $J_{k,\infty}$, and all its isolated singular points  of type  $A_k$.
\end{proposition}

\begin{proof}

The generic transversal type is  $A_{1}$ by hypothesis,  with transversal Milnor number  $\mu^{\pitchfork}=1$. Then by \eqref{eq:mdefect} the condition $\alpha_{p}(V, H)=1$ is equivalent to $\mu (V \cap \cH,p) = 2$, which means that  the hypersurface germ $(V \cap \cH,p)$ is an $A_2$-singularity. One may then apply the local classification of line singularities  \cite{Si-ILs} and obtain, firstly  that the local function defining $(V, p)$ is $\cR$-equivalent to 
$y^2 g(x,y)  + z_1^2 + \cdots+ z_{n-2}^2$, with $g(0,0)=0$,  and secondly that $ y^2 g(x,y$) is $\cR$-equivalent to $y^2 (y-x^k)$ for some $k$. This shows ou first claim.

To show our second claim, we  apply Corollary \ref{c:admissibleandspecial} as in the first paragraph of this subsection,  together with the result that we have just proved.
\end{proof}

\subsection{On projective cones and other hypersurfaces with polar degree zero}\label{ss:notcone}
Let $V \subset \bP^n$ be defined by $f(x_0,x_1,\cdots, x_n) = g(x_1,\cdots, x_n) =0$. This is one of the possible definitions of a projective cone $V$ over the hypersurface $V' \subset \bP^{n-1}$  given by $g(x_1,\cdots, x_n) =0$.   The point $q= [1:0: \cdots :0]$ is called the apex  of the cone. From the definiton \eqref{eq:grad} it follows that $\pol(V)=0$.\\
Let $V'$ have isolated singularities only, with Milnor numbers $\mu_1,\cdots,\mu_r$. Then $\Sing(V) = \cup_i^r \Sigma^c_i$, where each $\Sigma^c_i$ is a projective line with transversal type $\mu_i$, and such that all these lines intersect only at the apex $q$.

Let us compute $\pol(V)$ with formula \eqref{eq:dpgen}.  First observe that $g_i=0$, $\gamma_i=1$ and $\deg \Sigma^c_i=1$. There are no isolated critical points. The only special point is the apex $q$, where a local affine equation is given by $g=0$.
As  computed in \cite{Si-cambridge}, one has  $\chi (\cA_q)= (-1)^n ((d-1)^n - d \sum_i \tmu_i )$.   We therefore get $\pol(V) = (d-1)^n - d\sum_i \tmu_i - (d-1)^n + d\sum_i \tmu_i = 0$, which is indeed what was expected since $V$ is a cone.  

\smallskip

What can be said about other hypersurfaces with polar degree zero?  Historically, Hesse claimed that hypersurfaces with vanishing Hessian are always projective  cones. In 1875 Gordan and Noether \cite{GN76} gave several examples with polar degree zero but not cones. All known examples seem to have a singular locus of dimension at least $2$. It follows from 
the lower bound results \cite{ST-polcycles} that the polar degree zero can not occur  if $V$ contains isolated singularities. The following  question is open:

\smallskip

\emph{Do there exist hypersurfaces with 1-dimensional critical set and polar degree 0?}

\smallskip
We have the following  partial result, by \cite[Corollary 6.9]{ST-polcycles}: \emph{Let  $V$ be not a cone and $\pol(V) =0$. Then  $V$ has no special point. In particular $V$ has no isolated singularity (besides its non-isolated singularities).}


\end{document}